\pdfoutput=1
\documentclass[12pt, oneside]{amsart}

\usepackage{hyperref}
\hypersetup{%
	pdfpagemode={UseOutlines},
	bookmarksopen,
	pdfstartview={FitH},
	colorlinks,
	linkcolor={blue},
	citecolor={blue},
	urlcolor={blue}
}

\usepackage{geometry} 

\usepackage{dcolumn}
\newcolumntype{d}{D{.}{.}{-1} } 

\newcommand\salta[1]{}

\usepackage{amsmath, amssymb, amsthm}
\usepackage{mathrsfs}
\usepackage{aligned-overset}
\usepackage{mathtools}
\usepackage[numbers]{natbib}
\usepackage{xcolor}

\numberwithin{equation}{section}

\newtheorem{theorem}{Theorem}[section]

\newtheorem{lemma}[theorem]{Lemma}

\theoremstyle{remark}
\newtheorem{remark}[theorem]{Remark}

\newcommand{\R}{\mathbb{R}} 
\newcommand{\N}{\mathbb{N}} 
\newcommand{\Fock}{\mathcal{F}} 
\newcommand{\C}{\mathbb{C}} 
\renewcommand{\H}{\mathcal{H}} 
\newcommand{\G}{\mathcal{G}} 
\renewcommand{\S}{\mathcal{S}} 
\newcommand\eps\varepsilon
\renewcommand\k{\textsc{k}} 
\newcommand{\m}{\textsc{m}}
\renewcommand{\j}{\textsc{j}}

\newcommand\UU{{\mathcal U}}
\renewcommand\Re{\operatorname{Re}}

\newcommand{\commFe}[1]{{\color{blue} \textbf{*** {#1} ***}}}
\newcommand{\commFa}[1]{{\color{red} \textbf{*** {#1} ***}}}

\title[Existence of extremizers for partial sums of Toeplitz eigenvalues]{On the existence of extremizers for the sum of eigenvalues of Toeplitz operators}
\author{Fabio Nicola}
\address[Fabio Nicola]{Dipartimento di Scienze Matematiche, Politecnico di Torino, Corso Duca degli Abruzzi 24, 10129 Torino, Italy}
\email{fabio.nicola@polito.it}

\author{Federico Riccardi$^*$}
\address[Federico Riccardi]{Dipartimento di Scienze Matematiche, Politecnico di Torino, Corso Duca degli Abruzzi 24, 10129 Torino, Italy}
\email{federico.riccardi@polito.it}

\author{Paolo Tilli}
\address[Paolo Tilli]{Dipartimento di Scienze Matematiche, Politecnico di Torino, Corso Duca degli Abruzzi 24, 10129 Torino, Italy}
\email{paolo.tilli@polito.it}

\begin{document}
	
	\thanks{$^*$Corresponding author}
	
    
    \begin{abstract}
        \noindent We prove that,
        among all measurable sets $\Omega\subset\C$ of prescribed Lebesgue measure,
        there exists a set maximizing the sum of the first $\k$ eigenvalues ($\k\geq 1$) of the associated Toeplitz operator on the Fock space. In the Fock setting, the case $\k=1$ is well known, the optimal sets being balls of prescribed measure, whereas for $\k>1$ the existence of optimal sets appears to be new (maximizers are not known explicitly,
        and the optimality of balls remains conjectural). Moreover, under mild assumptions, our proof extends 
        to localization operators associated with abstract wavelet transforms.
        In this broader setting, the result is new even for $\k=1$. As an application, 
        we prove the existence of  optimal sets for the Donoho--Stark concentration problem and its generalization to orthonormal systems.
    \end{abstract}
    
    \maketitle

    \allowdisplaybreaks
    
\section{Introduction}

    In \cite{nicola_riccardi_tilli_partial_sums} it was proved that if a set $\Omega \subset \C$ is
    \emph{radially symmetric} and has finite measure,  then for every $\k \geq 1$ 
    \begin{equation*}
        \sum_{k=1}^{\k} \lambda_k(\Omega) \leq \sum_{k=1}^{\k} \lambda_k(\Omega^*),
    \end{equation*}
    where $\{\lambda_k(\Omega)\}_{k\geq 1}$ are the eigenvalues of the Toeplitz operator $T_{\Omega}$ on the Fock space, and $\Omega^*$ denotes the ball with the same measure as $\Omega$. Thus, among radially symmetric sets of prescribed measure, balls maximize the sum of the first $\k$ eigenvalues.
    If the radial symmetry assumption is removed, however,
    then for $\k=1$ it is known \cite{nicolatilli_fk} that balls remain optimal, whereas
    for $\k>1$ it is not even known whether an optimal set exists. In this paper we answer this existence question
by proving the following result (see Section \ref{sec:notation and setting} for details on the definitions).
    \begin{theorem}\label{th:main theorem for sets}
        Let $s>0$ and $\k \in \N$. Then the supremum
        \begin{equation}\label{eq:supremum sum eigenvalues}
            \sup \left\{ \sum_{k=1}^{\k} \lambda_k(\Omega) \colon |\Omega| = s\right\}
        \end{equation}
        is attained.
    \end{theorem}
    The uniqueness of optimal sets up to rigid motions and the optimality of balls in the general case
    (which we conjecture)
    remain open questions.

    To explain why partial sums of eigenvalues are meaningful in this context, we need to briefly recall the definition of the Fock space and of Toeplitz operators. The Fock space $\Fock$ is the Hilbert space of entire functions on $\C$ that are square-integrable with respect to the Gaussian measure $e^{-\pi|z|^2} dA(z)$. An important class of operators on the Fock space consists of Toeplitz operators $T_{\Omega}$ associated with some subset $\Omega \subset \C$ with finite Lebesgue measure $|\Omega|$. Since these operators are compact and positive, they admit a nonincreasing sequence of eigenvalues
    \begin{equation*}
        \lambda_1(\Omega) \geq \lambda_2(\Omega) \geq \cdots > 0.
    \end{equation*}
    Now, fix an integer $\k\geq 1$ and consider an orthonormal system  
    $F = (f_1, \ldots,f_{\k})$ of $\k$ functions in the Fock space. By the definition of $T_{\Omega}$ it follows that
    \begin{equation}\label{eq:variational connection}
        \sum_{k=1}^{\k} \langle T_{\Omega}f_k, f_k \rangle = \sum_{k=1}^{\k} \int_{\Omega} |f_k(z)|^2 e^{-\pi |z|^2} \, dA(z)
    \end{equation}
    and this sum
    (the trace of $T_\Omega$ restricted to $\mathop{\rm span}\{f_1,\ldots,f_{\k}\}$)
is the energy fraction of the orthonormal system $F$ inside $\Omega$. 
Since the functions in $F$ are entire and pairwise orthogonal, one expects
that this quantity (which cannot exceed $\k$) has some nontrivial upper bound in terms of the measure $s$ of $\Omega$. The connection with the spectrum of $T_\Omega$ is clear:
among all orthonormal systems $F$, the left-hand side of \eqref{eq:variational connection} 
is maximized when $F$ consists of the first $\k$ eigenfunctions of $T_{\Omega}$, and the maximal value 
is then the sum 
    \begin{equation}\label{eq:partial sum of eigenvalues}
        \sum_{k=1}^{\k} \lambda_k(\Omega)
    \end{equation}
    of the first $\k$ eigenvalues of $T_{\Omega}$, which thus
    quantifies the extent to which an orthonormal system of $\k$ functions can be concentrated on the set $\Omega$.
    It is then natural to investigate the existence of sets which, among all sets of prescribed measure, achieve optimal concentration and hence maximize the sum in  \eqref{eq:partial sum of eigenvalues},
    and Theorem \ref{th:main theorem for sets} provides a positive answer to this question. 
    
    Indeed, the case $\k=1$ ---which, as already mentioned, has been solved in \cite{nicolatilli_fk}--- fits into a broad
    line of research on many problems considered in the literature. An important connection is with the Wehrl entropy conjecture \cite{wehrl_entropy} (which can be interpreted as the ``global counterpart'' of the notion of concentration 
    that we consider in this paper) together with its many developments (see for example \cite{frank2023generalized}). Other relevant connections are with similar problems in settings other than the Fock space, see e.g. \cite{garcia_ortega_stability, frank2023sharp, gomez_kalaj_melentijevic_ramos, kalaj1, kulikov, ortega, nicola_stabilizer, ramos-tilli}, and maximization problems for the norms of Toeplitz operators with arbitrary symbols (not necessarily characteristic functions), see e.g. \cite{guo_lin_zhao_restricted_holder, nicolatilli_norm, riccardi_optimal_estimate}. Finally, further connections arise with concentration problems of various time-frequency representations (see e.g. \cite{nicola_romero_trapasso_existence, stra_svela_trapasso_existence, stra_svela_trapasso_existence_2, svela_trapasso_quantum_existence}) and, more generally,
    with \emph{shape optimization} problems for the eigenvalues of the Laplace operator (see e.g. \cite{bucur_buttazzo_book,henrot}).

    If we look at Theorem \ref{th:main theorem for sets} from a slightly different point of view, some possible connections with other problems emerge. In particular, building upon the min-max principle and the variational connection \eqref{eq:variational connection} we can state our theorem equivalently in the following way.
    \begin{theorem}\label{th:main theorem for functions and sets}
        Let $\k \in \N$ and $s>0$. Then the supremum
        \begin{equation}
        \label{supThm12}
            \sup \left\{ \sum_{k=1}^{\k} \int_{\Omega} |f_k(z)|^2 e^{-\pi |z|^2} \, dA(z) \colon |\Omega|=s,\, f_1, \ldots, f_{\k} 
            \textrm{ orthonormal in $\Fock$} \right\}
        \end{equation}
        is attained.
    \end{theorem}
    This formulation suggests possible connections with other problems involving functional inequalities for orthonormal systems. Some examples include the Strichartz inequality for orthonormal functions \cite{frank_lewin_lieb_seiringer} and $L^p$ bounds for Riesz and Bessel potentials of orthonormal functions \cite{lieb_Lp_bound}. It is plausible that some of the ideas developed in this paper can be transferred to those settings, and we plan to investigate these issues in a subsequent work.

In fact, even though for simplicity the
concentration problem \eqref{eq:supremum sum eigenvalues} has so far been considered only in the particular setting of the Fock space,
in Section \ref{sec:extension to the abstract setting} we extend Theorem \ref{th:main theorem for sets}
to a much more general setting, where Toeplitz operators in the Fock space are replaced with localization operators associated with an abstract wavelet transform (Theorem \ref{th:main theorem abstract setting}). In this abstract setting,
and in many concrete instances that it embraces (see Section \ref{sec:examples}), 
the existence of optimal sets is new even  in the case $\k = 1$.

A particularly relevant example is that of the Paley--Wiener space
        \begin{equation*}
            PW \coloneqq \{f \in L^2(\R) \colon \mathrm{supp}\, \widehat{f} \subseteq [-\pi,\pi] \}
        \end{equation*}
in connection with the well-known Donoho--Stark conjecture \cite{donoho_stark}, which appears to have been settled very recently \cite{abreu_speckbacher_donoho_stark}. This conjecture states that, among all
$f\in PW$ with $\Vert f\Vert_2=1$ and all sets $\Omega\subset\R$ of given Lebesgue measure $|\Omega|=s$, the supremum
\begin{equation}
    \label{supPW}
 \sup\int_\Omega |f(x)|^2\,dx
\end{equation}
is obtained when $\Omega$ is an \emph{interval} of length $s$ and $f$ is the first eigenfunction
(prolate spheroidal wave function) of the corresponding Slepian--Pollak operator. Since for every fixed $\Omega$
the supremum in \eqref{supPW} is equal to the first eigenvalue $\lambda_1(L_{\Omega})$ of the associated localization operator
$L_\Omega$, the Donoho--Stark conjecture  can be reformulated as follows: among all sets $\Omega$ of given measure $s>0$,
 intervals of length $s$ maximize $\lambda_1(L_\Omega)$. As a corollary of our result, we obtain:
\begin{theorem}[Existence of optimal sets for the Donoho-Stark Conjecture]\label{th:theorem PW introduction} Given $s>0$, among all sets $\Omega\subset\R$
of measure $s$ there exist sets maximizing the first eigenvalue $\lambda_1(L_\Omega)$. Moreover, any optimal
set is, up to sets of measure zero, a finite union of intervals.
\end{theorem}

To prove Theorem \ref{th:main theorem for functions and sets}, it is convenient
to use the bathtub principle (see \eqref{eq:bathtub principle}) to eliminate
the dependence on $\Omega$ and regard the 
orthonormal system $F=(f_k)$ as the main unknown. In this formulation (Theorem \ref{th:main theorem for J_s}),
a new set of ideas and tools, developed in Section \ref{sec:preliminary lemmas}, is needed 
in order to apply successfully the direct method
(maximizing sequence $\{F^{(n)}\}$ of orthonormal systems, weak limit $F^{(n)}\rightharpoonup F$, semicontinuity).

Indeed, the normalization and the orthogonality constraints
are not generally preserved  under weak convergence.
As a consequence, if the limit system
$F=(f_k)$ is not orthonormal, its concentration over admissible sets $\Omega$
no longer has the desired variational interpretation.
To overcome these difficulties, we introduce
Lemma \ref{lem:change the limit to make it orthogonal} and
Lemma \ref{lem:change the tail without changing the functional} 
(a geometric lemma valid in any Hilbert space),
which allow us to recover orthonormality in a suitable sense, together with the
disintegration Lemma \ref{lem:telescopic sum of the norms},
which allows us to split the concentration of $F$ into a superposition of concentrations of orthonormal systems of lower cardinality. In the end, we obtain that any maximizing sequence $F^{(n)}$ in fact converges strongly (up to translations) to an orthonormal system $F$, which maximizes the concentration.

This existence proof is presented in Section \ref{sec:proof of the main theorem} in the Fock setting. This approach allows one to focus on the genuine difficulties of the problem and on the new ideas and tools introduced in the paper, without being distracted by minor technicalities (such as the uniform continuity of Husimi functions or their vanishing at infinity), which are straightforward to verify in the Fock setting.

In Section \ref{sec:extension to the abstract setting} we then extend our result to the general framework of localization operators associated with an abstract wavelet transform, under the mild assumption \eqref{assumption:localization operators are positive}.
After introducing the abstract setting and reformulating the concentration problem appropriately, we revisit the key steps of the proof and explain how the argument adapts to the abstract framework. In particular, we show that several properties that are obvious or well known in the Fock setting remain valid in the abstract setting.


   \if \commFa{A number of important questions in harmonic analysis and mathematical physics center on the notion of localization, as it arises in the different manifestations of the uncertainty principle,} and one of the many facets of this everlasting problem concerns determining the optimal concentration for certain classes of functions, especially entire functions. In this paper we mainly focus on the Fock space and on the measurement of concentration via the eigenvalues of Toeplitz operators in this space. \commFa{We then generalize the results established for the Fock space to the setting of the wavelet transform associated with a general square-integrable projective unitary representation of a locally compact group.}   
    
    The Fock space $\Fock$ is the Hilbert space of entire functions on $\C$ that are square-integrable with respect to the Gaussian measure $e^{-\pi|z|^2} dA(z)$ (see Section \ref{sec:notation and setting} for definitions and notation). An important class of operators in this context is that of Toeplitz operators associated with some subset $\Omega \subset \C$ with finite Lebesgue measure $|\Omega|$, which we shall denote by $T_{\Omega}$. Since these operators are compact and positive, they admit a nonincreasing sequence of eigenvalues that we shall denote by
    \begin{equation*}
        \lambda_1(\Omega) \geq \lambda_2(\Omega) \geq \cdots > 0.
    \end{equation*}
    The connection between these eigenvalues and the concentration of functions in the Fock space can be illustrated as follows. Fix an integer $\k \in \N$ and a real number $s>0$. Then, consider a subset $\Omega \subset \C$ such that $|\Omega|=s$ and an orthonormal system of $\k$ functions $F = (f_1, \ldots,f_{\k}) \in \Fock^{\k}$. By the definition of $T_{\Omega}$ it follows that
    \begin{equation}\label{eq:variational connection}
        \sum_{k=1}^{\k} \langle T_{\Omega}f_k, f_k \rangle = \sum_{k=1}^{\k} \int_{\Omega} |f_k(z)|^2 e^{-\pi |z|^2} \, dA(z).
    \end{equation}
    The right-hand side represents the fraction of the energy of the functions in the orthonormal system $F$ inside $\Omega$, whereas the left-hand side is closely related to the first $\k$ eigenvalues of $T_{\Omega}$. In fact, among all possible orthonormal systems $F$, the left-hand side is maximized by choosing the first $\k$ eigenfunctions of $T_{\Omega}$ and the maximal value is exactly the sum of the first $\k$ eigenvalues of $T_{\Omega}$. Therefore, the partial sum
    \begin{equation}
        \sum_{k=1}^{\k} \lambda_k(\Omega)
    \end{equation}
    measures the extent to which an orthonormal system can be concentrated on the set $\Omega$. In light of this interpretation, a natural question is how large this partial sum can be, since
    \begin{equation}\label{eq:supremum partial sum}
        \sup \left\{ \sum_{k=1}^{\k} \lambda_k(\Omega) \colon |\Omega|=s \right\}
    \end{equation}
    measures how well an orthonormal set of $\k$ functions can be concentrated on a set of measure $s$. One may also ask which sets, if any, attain this supremum, thus realizing the best possible concentration. For $\k=1$, that is for the first eigenvalue, these problems were settled in \cite{nicolatilli_fk}, where it was proved that
    \begin{equation}
        \lambda_1(\Omega) \leq \lambda_1(\Omega^*),
    \end{equation}
    where $\Omega^* \subset \C$ is the ball centered at the origin such that $|\Omega|=|\Omega^*|$, and that equality is attained if and only if $\Omega$ is (equivalent up to a set of measure 0) a ball. On the other hand, for $\k > 1$ the authors proved in \cite{nicola_riccardi_tilli_partial_sums} that an analogous result holds, that is
    \begin{equation}
        \sum_{k=1}^{\k} \lambda_k(\Omega) \leq \sum_{k=1}^{\k} \lambda_k(\Omega^*),
    \end{equation}
    under the assumption that $\Omega$ is radially symmetric, while leaving open the existence question in the general case ---that is, without any symmetry assumption. The goal of this paper is to answer this question affirmatively, thus proving that an optimal set exists for every $\k \in \N$ and $s>0$.
    \begin{theorem}\label{th:main theorem for sets}
        Let $\k \in \N$ and $s>0$. Then, the supremum \eqref{eq:supremum partial sum} is attained.
    \end{theorem}
    Given the variational connection \eqref{eq:variational connection}, Theorem \ref{th:main theorem for sets} can be equivalently stated as follows.
    \begin{theorem}\label{th:main theorem for functions and sets}
        Let $\k \in \N$ and $s>0$. Then, the supremum
        \begin{equation*}
            \sup \left\{ \sum_{k=1}^{\k} \int_{\Omega} |f_k(z)|^2 e^{-\pi |z|^2} \, dA(z) \colon |\Omega|=s,\, f_1, \ldots, f_{\k} \in \Fock \textrm{ orthonormal system} \right\}
        \end{equation*}
        is attained.
    \end{theorem}
    In Section \ref{sec:notation and setting} we fix the notation and the framework of the problem. Then, in Section \ref{sec:preliminary lemmas} we prove some lemmas (which may be of independent interest, in particular Lemma \ref{lem:change the tail without changing the functional}) that will be helpful in the proof of the main theorems, which is given in Section \ref{sec:proof of the main theorem}. We do not directly prove Theorem \ref{th:main theorem for sets}, but rather the equivalent statement in Theorem \ref{th:main theorem for J_s} in which only orthonormal systems in $\Fock$ are involved and the corresponding optimal sets arise consequently. \commFe{Questa parte l'ho scritta rapidamente e solo per dare un'idea; è sicuramente da migliorare} The proof is based on the direct method of the calculus of variations, thus considering a maximizing sequence of orthonormal sets $F^{(n)} \in \Fock^{\k}$ and proving that it converges strongly to an orthonormal set $F \in \Fock^{\k}$. The main difficulties that arise are due to the translation invariance of the Fock space under the Weyl unitary operators \eqref{eq:Weyl unitary operators} and the fact that orthonormality is not preserved under weak convergence. We overcome this problem with the aid of the lemmas in Section \ref{sec:preliminary lemmas} and a careful analysis of the right-hand side of \eqref{eq:variational connection} that takes into account the splitting of the maximizing sequence $F^{(n)}$ into its converging part and its component that escapes at infinity.

    In conclusion, in Section \ref{sec:extension to the abstract setting} we extend our results in a much more general setting. In fact, a careful analysis of the lemmas in Section \ref{sec:preliminary lemmas} and the proof in Section \ref{sec:proof of the main theorem} reveals that the proof still works with minor adjustments if, instead of Toepltiz operators, one considers localization operators associated with an abstract wavelet transform satisfying the mild assumption \eqref{assumption:localization operators are positive}. Then, we show several examples in which our result appplies. We point out that in some of these examples the existence of optimal sets is new even for the first eigenvalue, so with $\k = 1$.
        
    \commFa{Magari da qualche parte si possono aggiungere riferimenti ai problemi di ottimizzazione per sistemi ortonomali --- Lieb, Frank, etc.} \commFe{Intendi nell'introduzione? Oppure si riesce addirittura ad aggiungere una sottosezione nella sezione 4?}
    
    \commFa{Inoltre enfatizzerei anche la connessione con la problematica della congettura di Wehrl, di cui questi problemi possono vedersi come una versione locale} \commFe{Ottima idea, bisogna però pensare un attimo su come inserirlo visto che (se non sbaglio) la connessione c'è ``solo'' per $\k=1$. In ogni caso adesso bisognerà trovare un escamotage per citare i molti paper che sono collegati con questo argomento.} \fi

\section{Preliminary results}\label{sec:preliminary results}

\subsection{Notation and setting}\label{sec:notation and setting}

The Fock space, denoted by $\Fock$, is the Hilbert space of all entire functions $f \colon \C \to \C$ such that
\begin{equation*}
    \|f\|^2 = \int_{\C} |f(z)|^2 e^{-\pi |z|^2} \, dA(z) < \infty,
\end{equation*}
where $dA(z)$ denotes the Lebesgue measure on $\C$,
with the natural inner product
\[
\langle f,g\rangle =\int_{\C} f(z)\overline{g(z)} e^{-\pi |z|^2} \, dA(z)
\]
inherited from $L^2=L^2(\C, e^{-\pi |z|^2} dA(z))$.
Assuming familiarity with the basic theory of Fock spaces 
(see \cite{zhu_book} for details), the main objects of interest for our purposes
are the 
Weyl unitary operators $U_w:\Fock\to\Fock$, defined for $w\in\C$ by
\begin{equation}\label{eq:Weyl unitary operators}
    U_wf(z) = f(z-w) e^{\pi z \overline{w}} e^{-\pi |w|^2/2}, \quad z \in \C,
\end{equation}
and the \emph{Husimi function}
\begin{equation}\label{eq:definition Husimi function}
    u_f(z) \coloneqq |f(z)|^2 e^{-\pi |z|^2},    \quad z \in \C
\end{equation}
associated with $f\in\Fock$ or, more generally, the \emph{joint Husimi function}
\begin{equation}\label{eq:defjointHusimi}
    u_F(z) \coloneqq \sum_{k=1}^{\k}|f_k(z)|^2 e^{-\pi |z|^2},
    \quad F=(f_1,\ldots,f_\k)\in \Fock^\k
\end{equation}
associated with a $\k$-tuple $F$ (we use the notation $u_F$ 
for arbitrary tuples, although $F$ will typically be an orthonormal set). We endow the product space $\Fock^{\k}$ with the natural norm
\begin{equation}
\label{normaF}
\Vert F\Vert^2:=\sum_{k=1}^\k \Vert f_k\Vert^2,
\quad F=(f_1,\ldots,f_\k) \in \Fock^{\k}.
\end{equation}

The well-known identity
\begin{equation}\label{eq:translation of the Husimi}
    |U_wf(z)|^2 e^{-\pi |z|^2} = |f(z-w)|^2 e^{-\pi |z-w|^2},\quad f\in \Fock
\end{equation}
reveals that the unitary operators $U_w$ induce a pure shift
$u_{U_w f}(z)=u_f(z-w)$
on the Husimi function of $f$. The same is true for the joint Husimi function
of a tuple 
$F$, namely
\begin{equation}
    \label{shiftH}
u_{U_w F}(z)=u_F(z-w)
\end{equation}
where $F=(f_1,\ldots,f_\k)$
and $U_wF=(U_w f_1,\ldots,U_w f_\k)$.

Some basic properties of the Husimi functions are integrability, in that
\begin{equation}\label{eq:normalization Husimi}
    \int_{\C} u_f(z) \, dA(z) = \|f\|^2,
\end{equation}
boundedness 
\begin{equation}\label{eq:boundedness Husimi}
    u_f(z) 
    \leq \|f\|^2 \|k_z \|^2 = \|f\|^2, \quad z \in \C,
\end{equation}
and vanishing at infinity
\begin{equation*}
    \lim_{|z| \to \infty} u_f(z) = 0.
\end{equation*}
It is clear that all these properties hold also for joint Husimi functions (with $\|f\|$ replaced with $\|F\|$). However, the boundedness estimate \eqref{eq:boundedness Husimi} holds in a stronger form if $F \in \Fock^{\k}$ is an orthonormal set, since in this case it holds
\begin{equation}\label{eq:boundedness Husimi orthonormal set}
    u_F(z) \leq 1, \quad z \in \C.
\end{equation}

If  $P \colon L^2 \to \Fock$ denotes the orthogonal projection onto the Fock space
and  $\Omega \subset \C$ is measurable,
the \emph{Toeplitz operator} on $\Omega$ is defined as
\begin{equation*}
    T_{\Omega} = P \chi_{\Omega} \colon \Fock \to \Fock
\end{equation*}
where $\chi_{\Omega}$ denotes the characteristic function of $\Omega$. 
The associated quadratic form is
\begin{equation}\label{eq:quadratic form Toeplitz operator}
    \langle T_{\Omega} f, f \rangle = \int_{\Omega} |f(z)|^2 e^{-\pi |z|^2}\, dA(z) 
    = \int_{\Omega} u_f(z) \, dA(z),\quad f\in\Fock
\end{equation}
is positive
and, when $\Omega$ has finite measure
$|\Omega|<\infty$, the operator $T_{\Omega}$ is 
compact (in fact trace-class, see \cite{zhu_book}) and positive. Therefore, it admits a nonincreasing sequence of positive eigenvalues, which we will denote by
\begin{equation*}
    1 >  \lambda_1(\Omega) \geq \lambda_2(\Omega) \geq \cdots > 0.
\end{equation*}
Given $\k\geq 1$, our goal is to maximize the sum of the first $\k$ eigenvalues 
\begin{equation}\label{eq:sum of eigenvalues}
    \sum_{k=1}^{\k} \lambda_k(\Omega)
\end{equation}
(also known as the Ky Fan $\k$-norm of $T_{\Omega}$)
among all sets $\Omega$ with prescribed measure $s>0$. For fixed $\Omega$,
by the min-max principle this sum is equal to the maximum of
\begin{equation}\label{eq:variational characterization}
    \sum_{k=1}^{\k} \langle T_{\Omega}f_k,f_k \rangle
    = \int_{\Omega} u_F(z) \, dA(z)
\end{equation}
among all orthonormal systems $F = (f_1, \ldots,f_{\k}) \in \Fock^{\k}$, where 
$u_F$ is the joint 
Husimi function of the tuple $F$ as in \eqref{eq:defjointHusimi}.
The maximum is achieved when $f_1,\ldots,f_{\k}$ are the first $\k$ eigenfunctions of $T_{\Omega}$, and hence for every $s>0$ the supremum in \eqref{eq:supremum sum eigenvalues}
coincides with
\begin{equation}
    \S_{\k}(s) \coloneqq \sup\left\{  \int_\Omega u_F(z)\,dA(z)\,:\,\, \text{$|\Omega|=s$,\,\, 
    $F\in\Fock^\k$ orthonormal system}\right\}.
\end{equation}
This shows that Theorem \ref{th:main theorem for functions and sets}
is equivalent to Theorem \ref{th:main theorem for sets}.
On the other hand, for every orthonormal system $F$,
by the bathtub principle \cite[][Theorem 1.14]{liebloss}
(see also \cite{nicolatilli_fk})
\begin{equation}\label{eq:bathtub principle}
    \int_{\Omega} u_F(z) \, dA(z) \leq \int_{\Omega_F(s)} u_F(z) \, dA(z)\quad
    \text{for every $\Omega$ such that $|\Omega| = s$,}
\end{equation}
with equality if and only if $\Omega = \Omega_F(s)$ (up to sets of measure 0).
Here and throughout, $\Omega_F(s)$ denotes the unique \emph{superlevel set} of $u_F$ of measure
$s$, that is a set of the form
\[
\{u_F>t\}:=\left\{z\in\C\,:\,\, u_F(z)>t\right\}
\]
where $t>0$ is chosen so as to guarantee the measure matching $|\{u_F>t\}|=s$.

More precisely, letting $\mu_F(t) = | \{u_F > t\}|$ $(t>0)$ denote the distribution function of $u_F$, it is well known (see e.g. \cite{nicolatilli_fk,nicola_riccardi_tilli_partial_sums})
that $\mu_F$ is continuous, strictly decreasing, and hence invertible on $(0,\max u_F]$, with values onto $[0,\infty)$.
Thus, denoting by $\mu_F^{-1}:[0,\infty)\to (0,\max u_F]$ the inverse function,
the superlevel set
\begin{equation}\label{defOmegaFs}
    \Omega_F(s) \coloneqq \{ u_F > \mu_F^{-1}(s)\} \subset \C,
\end{equation}
has Lebesgue measure equal to $s$.

 Thus, summing up, for every $F \in \Fock^{\k}$ the integral on the right-hand side of \eqref{eq:variational characterization} is maximized when $\Omega$ is the unique superlevel set of $u_F$ with measure $s$, that is $\Omega_F(s)$.
 Therefore, 
defining the
 \emph{concentration functional} $J_s \colon \Fock^{\k} \to [0,\infty)$ as
\begin{equation}\label{eq:definition of J_s}
    J_s(F) = \int_{\Omega_F(s)} u_F(z) \, dA(z), \quad F \in \Fock^{\k},
\end{equation}
it is clear that Theorem \ref{th:main theorem for functions and sets} 
(and hence also Theorem \ref{th:main theorem for sets})
is equivalent to the following statement, which will be proved in Section \ref{sec:proof of the main theorem}.
\begin{theorem}\label{th:main theorem for J_s}
    Let $s>0$ and $\k\geq 1$. Then, the supremum
    \begin{equation}\label{eq:supremum for J_s}
        \S_{\k}(s) = \sup \left\{ J_s(F) \colon F \in \Fock^{\k} \textrm{ is an orthonormal system} \right\}
    \end{equation}
    is attained by at least one orthonormal system.
\end{theorem}
Indeed, if an orthonormal set $F$ achieves the supremum in \eqref{eq:supremum for J_s}, then the same $F=(f_1,\ldots,f_\k)$, coupled with 
its superlevel set $\Omega=\Omega_F(s)$, achieve the supremum in 
\eqref{supThm12}; conversely, if  $F=(f_1,\ldots,f_\k)$ and $\Omega$
achieve the supremum in \eqref{supThm12}, then by \eqref{eq:bathtub principle}
the same $F$ is optimal in \eqref{eq:supremum for J_s}.
\begin{remark}\label{remark:optimal sets are bounded}
    In particular this implies that if an optimal set exists then, up to sets of measure zero, it must be bounded, as it must be the superlevel set of some joint Husimi function, which vanishes at infinity. 
\end{remark}

\begin{remark}\label{remark:translation invariance and monotonicity of J_s}
    By \eqref{eq:definition of J_s} and 
    \eqref{defOmegaFs},
    it is clear that the supremum in \eqref{eq:supremum for J_s}
 is an increasing function  of $s$. Moreover, although the size $\k$
 of the tuples $F$ is usually fixed, 
the supremum in \eqref{eq:supremum for J_s}
 is increasing with respect to $\k$ as well.
\end{remark}

We end this section by recording some further properties of the
functional $J_s$ which will be used in the sequel. 
First, the functional $J_s$ is invariant under the action of
the Weyl unitary operators 
\eqref{eq:Weyl unitary operators} on the tuple $F$, that is
    \begin{equation}
    \label{invarianceJ_s}
        J_s(F)=J_s(U_w F)\quad\forall w\in\C.
    \end{equation}
Indeed, by \eqref{eq:definition of J_s} and \eqref{defOmegaFs}, $J_s(F)$ is defined as the integral of $u_F$
over its superlevel set of measure $s$: the value of this integral is unchanged
if $u_F(z)$ is replaced with $u_F(z-w)$, and thus it coincides with $J_s(U_w F)$
by \eqref{shiftH}.

Moreover the functional $J_s$ is locally Lipschitz continuous
on the product space $\Fock^\k$ endowed with the natural norm \eqref{normaF}. More precisely, we have the following.
    \begin{lemma}
    \label{lem:continuity of J_s}
        Let $F=(f_1,\ldots,f_{\k})$, $G = (g_1, \ldots, g_{\k}) \in \Fock^{\k}$. Then
        \begin{equation}\label{eq:J_s is continuous}
            |J_s(F)-J_s(G)| \leq  (\|F\| + \|G\|) \|F-G\|.
        \end{equation}
    \end{lemma}
    
    \begin{proof}
        From \eqref{eq:definition of J_s} we have
        \begin{align*}
            J_s(F) 
                   &= \int_{\Omega_F(s)} \left( u_F(z) - u_G(z) \right) \, dA(z) + \int_{\Omega_F(s)} u_G(z) \, dA(z) \\
                   &\leq \int_{\Omega_F(s)} \left( u_F(z) - u_G(z) \right) \, dA(z) + \int_{\Omega_G(s)} u_G(z) \, dA(z) \\
                   &\leq  \int_{\C} \left| u_F(z) - u_G(z) \right| \, dA(z) + J_s(G),
        \end{align*}
        where the first inequality is 
        just the bathtub principle \eqref{eq:bathtub principle}
        applied to $u_G$. Interchanging the roles of $F$ and $G$ we obtain
        \begin{equation*}
            |J_s(F)- J_s(G)| \leq \int_{\C} \left| u_F(z) - u_G(z) \right| \, dA(z)
            \leq \sum_{k=1}^\k\int_{\C} \left| u_{f_k}(z) - u_{g_k}(z) \right| \,dA(z).
        \end{equation*}
       On the other hand, for every single $k$
        we have from Cauchy--Schwarz 
        \begin{align*}
\int_{\C}\left| f_k(z)-g_k(z)\right|\,
        \left| f_k(z)\right| e^{-\pi|z|^2}\,dA(z)\leq
        \Vert f_k-g_k\Vert \, \Vert f_k\Vert,\\
        \int_{\C}\left| f_k(z)-g_k(z)\right|\,
        \left| g_k(z)\right| e^{-\pi|z|^2} \,dA(z)\leq
        \Vert f_k-g_k\Vert \, \Vert g_k\Vert,
        \end{align*}
while        
        it is clear from \eqref{eq:definition Husimi function}
        that
        \[
        \int_{\C}
        \left| u_{f_k}(z) - u_{g_k}(z) \right|\,dA(z)
        \leq \int_{\C}\left| f_k(z)-g_k(z)\right|\,\left(
        \left| f_k(z)\right|+\left|g_k(z)\right|\right) e^{-\pi|z|^2} \,dA(z),
        \]
so that
        \[
        \int_{\C}
        \left| u_{f_k}(z) - u_{g_k}(z) \right|\,dA(z)
\leq \Vert f_k-g_k\Vert \, \left( \Vert f_k \Vert + \Vert g_k\Vert \right).
\]
Thus, summing over $k$ and using Cauchy--Schwarz in $\R^\k$ gives
\[
\sum_{k=1}^\k        \int_{\C}
\left| u_{f_k}(z) - u_{g_k}(z) \right|\,dA(z)
\leq \|F-G\| 
\left(\|F\|+\|G\|\right)
\]
so that, combining with the previous estimate gives the desired result.
\end{proof}

\subsection{Preliminary lemmas}\label{sec:preliminary lemmas}
In this section we gather several lemmas that are going to be useful in the proof of Theorem \ref{th:main theorem for J_s}. The first elementary lemma will allow us to modify the weak limit of a maximizing sequence for the functional $J_s$ \eqref{eq:definition of J_s} in order to turn it into an orthogonal (but not necessarily orthonormal) system without affecting the Husimi function.

\begin{lemma}\label{lem:change the limit to make it orthogonal}
For every  $F = (f_1, \ldots,f_{\k}) \in \Fock^{\k}$ there exists a unitary matrix $U \in \C^{\k \times \k}$ such that, if $G = UF$ 
is the $\k$-tuple $(g_1, \ldots, g_{\k}) \in \Fock^{\k}$ defined by
\begin{equation}\label{defUF}
     \begin{pmatrix}
         g_1(z) \\ \vdots \\ g_\k(z)
    \end{pmatrix}
:= U 
\begin{pmatrix}
         f_1(z) \\ \vdots \\ f_\k(z)
    \end{pmatrix},\quad z\in\C,
\end{equation}
then
\begin{enumerate}
    \item[(i)] the functions in $G=(g_1,\ldots,g_\k)$ are
    pairwise orthogonal;
    \item[(ii)] the joint Husimi function is preserved, that is
    $u_F(z)=u_G(z)$ for every $z\in\C$.
\end{enumerate}
Moreover,  if $F$ is an orthonormal system in $\Fock$, then so is $UF$,
for every unitary matrix $U$.

\end{lemma}
\begin{proof}
If $M = (\langle f_j, f_k \rangle) \in \C^{\k \times \k}$ is the Gram matrix associated with $F$ and  $U \in \C^{\k \times \k}$
is an arbitrary matrix, then the Gram matrix associated with $G=UF$
is $UMU^*$. Thus (i) is achieved, if $U$ is any unitary matrix  that diagonalizes $M$. 

 Moreover, for every unitary matrix $U$ and for every $z\in\C$,  by \eqref{defUF}
 the Euclidean norm of the vector $(g_k(z))\in\C^\k$
 coincides with that of $(f_k(z))$, so that $u_G(z)=u_F(z)$
 is immediate from \eqref{eq:defjointHusimi}.

 Finally, if $F$ is orthonormal, then its Gram matrix $M$ is the identity matrix $I$, so that the Gram matrix of $UF$ is $UMU^*=UU^*=I$ as well,
 and $UF$ is also orthonormal.
\end{proof}
Before stating the next lemma, we introduce a notation to denote normalized functions and tuples. In particular, for every $f \in \Fock$ we let
\begin{equation*}
    \widehat{f} \coloneqq \dfrac{f}{\|f\|},
\end{equation*}
with the understanding that $\widehat{f}=0$ if $f = 0$. Similarly, if $F = (f_1, \ldots, f_{\k}) \in \Fock^{\k}$ is a $\k$-tuple, we let
\begin{equation*}
    \widehat{F} = (\widehat{f}_1, \ldots, \widehat{f}_{\k}).
\end{equation*}
\begin{lemma}\label{lem:telescopic sum of the norms}
Let $F = (f_1,\ldots,f_{\k}), H=(h_1, \ldots, h_{\k}) \in \Fock^{\k}$ such that $\|f_k\|^2+\|h_k\|^2 = 1$ for every $k=1,\ldots,\k$. Assume that $F$ is ordered so that $\|f_1\| \geq \|f_2\| \geq \cdots \geq \|f_{\k}\|$ and let
\begin{equation*}
    a_0 = 1-\|f_1\|^2, \quad a_k = \|f_k\|^2-\|f_{k+1}\|^2 \ \  k=1,\ldots,\k-1, \quad a_{\k} = \|f_{\k}\|^2.
\end{equation*}
For every $0 \leq k \leq \k$ we introduce the mixed blocks
\begin{equation}\label{eq:definition mixed blocks}
    \widehat{M}_k=(\widehat{f}_1, \ldots, \widehat{f}_k, \widehat{h}_{k+1}, \ldots, \widehat{h}_{\k}) \in \Fock^{\k},
\end{equation}
with the understanding that $\widehat{M}_0 = \widehat{H}$ and $\widehat{M}_{\k} = \widehat{F}$. Then, it holds
\begin{equation}\label{eq:telescopic sum}
    u_F+u_H = \sum_{k=0}^{\k} a_k u_{\widehat{M}_k}
\end{equation}
\end{lemma}
\begin{proof}
We start noticing that
\begin{equation*}
    \sum_{k=0}^{\k} a_k = 1, \quad \|f_j\|^2 = \sum_{k=j}^{\k} a_k \quad j=1,\ldots,\k,
\end{equation*}
Moreover, since $\|h_j\|^2 = 1-\|f_j\|^2$ for every $j=1,\ldots,\k$, we also notice that 
\begin{equation*}
    \|h_j\|^2 = \sum_{k=0}^{j-1} a_k, \quad j=1,\ldots,\k.
\end{equation*}
Then, it holds
\begin{align*}
    u_F = \sum_{j=1}^{\k} u_{f_j} = \sum_{j=1}^{\k} \|f_j\|^2 u_{\widehat{f}_j} = \sum_{j=1}^{\k} \sum_{k=j}^{\k}a_k u_{\widehat{f}_j} = \sum_{k=1}^{\k} a_k \sum_{j=1}^{k} u_{\widehat{f}_j},
\end{align*}
and similarly
\begin{align*}
    u_H = \sum_{j=1}^{\k} u_{h_j} = \sum_{j=1}^{\k} \|h_j\|^2 u_{\widehat{h}_j} = \sum_{j=1}^{\k} \sum_{k=0}^{j-1}a_k u_{\widehat{h}_j} = \sum_{k=0}^{\k-1} a_k \sum_{j=k+1}^{\k} u_{\widehat{h}_j}.
\end{align*}
Hence, summing up we obtain
\begin{align*}
    u_F+u_H &= a_0 u_{\widehat{H}} + \sum_{k=1}^{\k-1} a_k \left(\sum_{j=1}^{k} u_{\widehat{f}_j} + \sum_{j=k+1}^{\k} u_{\widehat{h}_j}\right) + a_{\k} u_{\widehat{F}} \\
            &= \sum_{k=0}^{\k} a_k u_{\widehat{M}_k}.
\end{align*}
\end{proof}
\begin{lemma}\label{lem:continuity of the supremum}
    For every $\k \geq 1$ the function $s \mapsto \S_{\k}(s)$ is continuous.
\end{lemma}
\begin{proof}
    Let $F \in \Fock^{\k}$ be an orthonormal system. Then, since $0 \leq u_F \leq 1$ (recall \eqref{eq:boundedness Husimi orthonormal set}), for every $0 \leq t \leq s$ it holds
    \begin{equation*}
        J_s(F)-J_t(F) = \int_{\Omega_F(s) \setminus \Omega_F(t)} u_F(z) \, dA(z) \leq |\Omega_F(s) \setminus \Omega_F(t)| = s-t.
    \end{equation*}
    Taking the supremum with respect to $F$ leads to
    \begin{equation*}
        \S_{\k}(s) - \S_{\k}(t) \leq s-t,
    \end{equation*}
    which means that $\S_{\k}$ is 1-Lipschitz.
\end{proof}
\begin{lemma}\label{lem:convexity of the supremum}
    Let $\k \geq 1$. Assume that for every $1 \leq \m < \k$ and for every $s>0$ the supremum $\S_{\m}(s)$ is attained. Then, for every $1 \leq \m < \k$ and for every $s>0$ it holds
    \begin{equation}\label{eq:convexity of the supremum}
        \max_{0 \leq t \leq s} \left( \S_{\m}(t) + \S_{\k-\m}(s-t) \right) < \S_{\k}(s).
    \end{equation}
\end{lemma}
\begin{proof}
    We start noticing that the maximum in the left-hand side of \eqref{eq:convexity of the supremum} is well-defined because the function $t \mapsto \S_{\m}(t)+\S_{\k-\m}(s-t)$ is continuous.

    Then, for $t=0$ and $t=s$ it clearly holds
    \begin{equation*}
        \S_{\m}(s) < \S_{\k}(s), \quad \S_{\k-\m}(s) < \S_{\k}(s)
    \end{equation*}
    because we are assuming that $\S_{\m}(s)$ and $\S_{\k-\m}(s)$ are attained.

    Next, assume $t \in (0,s)$. We denote by $\Omega^*_{\m}(t)$ a set that achieves the supremum $\S_{\m}(t)$, and similarly $\Omega^*_{\k-\m}(s-t)$ denotes a set that achieves the supremum $\S_{\k-\m}(s-t)$. Since both $\Omega^*_{\m}(t)$ and $\Omega^*_{\k-\m}(s-t)$ are bounded (see Remark \ref{remark:optimal sets are bounded}), we can find $w \in \C$ such that
    \begin{equation*}
        \Omega^*_{\m}(t) \cap \left( \Omega^*_{\k-\m}(s-t) + w \right) = \emptyset.
    \end{equation*}
    Hence, exploiting the invariance under Weyl operators we have
    \begin{align*}
        \S_{\m}(t) + \S_{\k-\m}(s-t) &= \sum_{k=1}^{\m} \lambda_k(\Omega^*_{\m}(t)) + \sum_{k=1}^{\k-\m} \lambda_k(\Omega^*_{\k-\m}(s-t))\\
                                     &= \sum_{k=1}^{\m} \lambda_k(\Omega^*_{\m}(t)) + \sum_{k=1}^{\k-\m} \lambda_k(\Omega^*_{\k-\m}(s-t)+w)\\
                                     &< \sum_{k=1}^{\k} \lambda_k\left(\Omega^*_{\m}(t) \cup (\Omega^*_{\k-\m}(s-t)+w)\right) \\
                                     &\leq \S_{\k}(s),
    \end{align*}
    where strict inequality is a consequence of the Ky Fan's maximum principle (see \cite[][Problem I.6.15]{bhatia}) and of the fact that Toeplitz operators over sets of positive measure are positive operators, while the last inequality follows because the set $\Omega^*_{\m}(t) \cup (\Omega^*_{\k-\m}(s-t)+w)$ has measure $s$.
\end{proof}

\section{A lemma in Hilbert space}
The next lemma may be of independent interest.

\begin{lemma}\label{lem:change the tail without changing the functional}
Let $\H$ be a Hilbert space with $\dim \H\geq 2\k$ ($\k\geq 1$), and assume that $F=(f_1,\ldots,f_{\k})$, $G = (g_1, \ldots, g_{\k}) \in \H^{\k}$
are such that:
\begin{enumerate}
     \item[(a)]    $F+G$ is an orthonormal system.
    \item[(b)]    $F$ is an orthogonal system 
    with 
    $\|f_j\| \leq 1$ for every $j=1,\ldots,\k$.
    \item[(c)]  $|\langle g_j, f_k \rangle| \leq 
    \eps$ for some $\eps\in(0,1)$ and for every $j,k \in\{ 1,\ldots,\k\}$.
\end{enumerate}
Then, there exist $H = (h_1,\ldots,h_{\k}) \in \H^{\k}$ and a constant $C>0$, depending only on $F$ and $\k$ 
(in particular, independent of $G$ and $\eps$), such that
\begin{enumerate}
    \item[(i)] $F+H$ is an orthonormal system.
    \item[(ii)] $H$ is an orthogonal system and every element of $H$ is orthogonal to every element of $F$.
    \item[(iii)] \label{res: H and G are close} $\|h_j - g_j\| \leq C \sqrt{\varepsilon}$ for every $j\in\{1,\ldots,\k\}$. 
\end{enumerate}
\end{lemma}

\begin{proof}
The claim is trivial  if $F=(0,\ldots,0)$
(in that case one may take $H=G$ and $C=0$), 
hence we may assume that
$f_i\not=0$ for at least one index $i$, and define
\begin{equation}
\label{defC1}
C_1:=\sum_{\{i \colon f_i \neq 0\}} \dfrac{1}{\|f_i\|^2}.
\end{equation}
This is the first in a sequence of constants 
$C_1,C_2,\ldots$,  depending only on $F$ and $\k$, 
which will eventually determine the constant $C$ in (iii).

\smallskip

\textbf{Step A. } Preliminary estimates. 
First observe that assumptions (a) and (c) together imply that 
\begin{equation}\label{estgjgk}
    |\langle g_j, g_k \rangle| \leq 2 \varepsilon\quad\forall j\neq k.
\end{equation}
Then, letting
$\mathcal{W} = \mathrm{span}\{f_1, \ldots, f_{\k}\} \subset \H$
and
denoting by $P_{\mathcal{W}}$ and $P_{\mathcal{W}^{\perp}}$ the projections onto $\mathcal{W}$ and $\mathcal{W}^{\perp}$ respectively, we define the projected vectors
\begin{equation}\label{defvj}
    v_j := P_{{\mathcal{W}}^{\perp}} g_j = g_j - P_{{\mathcal{W}}} g_j, \quad j = 1,\ldots,\k.
\end{equation}
Every $v_j$ is close to the corresponding $g_j$. Indeed,
from assumption (c) and \eqref{defC1} 
\begin{align}\label{vcloseg}
    \|v_j-g_j\|^2 = \|P_{\mathcal{W}}g_j\|^2 = \sum_{\{i \colon f_i \neq 0\}} \frac {\left\vert \langle g_j, f_i\rangle
    \right\vert^2}
    {\|f_i\|^2}  \leq  C_1 \varepsilon^2,
\end{align}
where in the second equality we used the first part of assumption (b).
Moreover, the $v_j$'s  are  almost mutually 
orthogonal, just as the $g_j$'s are. In fact, from
\begin{equation*}
    \langle g_j, g_k \rangle = \langle P_{\mathcal{W}} g_j + P_{\mathcal{W}^{\perp}}g_j, P_{\mathcal{W}} g_k + P_{\mathcal{W}^{\perp}}g_k \rangle = \langle P_{\mathcal{W}} g_j, P_{\mathcal{W}} g_k \rangle + \langle v_j, v_k \rangle
\end{equation*}
we obtain 
\begin{equation}\label{estvjvk}
    | \langle v_j, v_k \rangle | \leq | \langle g_j, g_k \rangle| + | \langle P_{\mathcal{W}} g_j, P_{\mathcal{W}} g_k \rangle| \leq 2 \varepsilon + C_1 \varepsilon^2 \leq C_2 \varepsilon\quad\forall j\neq k,
\end{equation}
where  $C_2:=2+C_1$. In this estimate  we used \eqref{estgjgk} and
then Cauchy--Schwarz along with the estimates
on $\Vert P_{\mathcal W}g_j\Vert$  
and $\Vert P_{\mathcal W}g_k\Vert$  
that one infers from
\eqref{vcloseg}.

Next, recalling that $\|f_j\|^2\leq 1$,  we define the numbers
\begin{equation}
\label{defdk}
d_j:=1-\|f_j\|^2\geq 0,\quad j=1,\ldots,\k.
\end{equation}
Since $\|f_j+g_j\|^2=1$, we have
\begin{align*}
    &d_j-\|v_j\|^2
    = \|g_j\|^2
+2 \Re \langle f_j,g_j\rangle -\|v_j\|^2
= \|P_{\mathcal W}g_j\|^2+
2 \Re \langle f_j,g_j\rangle
\end{align*}
and hence
from \eqref{vcloseg} and
assumption (c) of the lemma we obtain
\begin{equation}
\label{estdjvj}
\left| \,d_j-\|v_j\|^2\,\right|
\leq
C_1\eps^2+2\eps\leq C_2\eps\quad\forall j\in\{1,\ldots,\k\}.
\end{equation}

\smallskip

\textbf{Step B.} We split the set of subscripts 
$\{1,\ldots,\k\}$ according to 
whether the norm of $v_k$ is ``big'' or ``small'',
letting
\begin{equation}\label{defBS}
    B= \left\{k  \colon \|v_k\|^2 \geq \k C_2 \varepsilon
    \right\}, \quad S = \{1,\ldots,\k\} \setminus B,
\end{equation}
where $C_2=2+C_1$ is the constant already used in
\eqref{estvjvk} and \eqref{estdjvj} (this threshold will ensure that the vectors $(v_j)_{j\in B}$ are linearly independent).

The construction of the
vectors $h_j$,  that appear in the claim of the 
lemma will be different according to whether $j\in B$
or $j\in S$, and will be performed in the next two steps
of the proof.

\medskip

\textbf{Step C.} Here we construct $h_j$ for $j\in B$.
If  $B=\emptyset$,  this step is empty and
the proof proceeds
directly to  Step D.
Hence, after relabelling the indices,
 we  may  assume
that $B = \{1, \ldots \m\}$ for some $\m\in\{1,\ldots, \k\}$.

Consider the finite rank operator 
\[
X:\C^\m\to {\mathcal W}^\perp,\quad Xe_j=v_j,
\quad j\in\{1,\ldots,\m\}
\]
where $\{e_j\}$ is the canonical basis of $\C^\m$,
and the associated Gram matrix
\[
A:=X^*X=\bigl(\langle v_j,v_k\rangle\bigr)_{j,k=1}^{\m}.
\]
From \eqref{estvjvk} and the definition of $B$ in \eqref{defBS} it is clear that
$A$
is strictly diagonally 
dominant, namely
\[
A_{jj}=\| v_j\|^2\geq \k C_2\eps
\quad\text{while}\quad
\sum_{k\neq j}|A_{jk}|
=\sum_{k\neq j}|\langle v_j,v_k\rangle|\leq (\m-1)C_2\eps
<\k C_2\eps.
\]
In particular, 
the vectors $(v_j)_{j \in B}$ are linearly independent so
$X$ is injective and, if $X=UA^{1/2}$ is the polar decomposition of $X$, $U:\C^\m\to{\mathcal W}^\perp$
is an isometry, that is, $U^*U=I_{\C^\m}$.

Introducing the $\m\times\m$ diagonal matrix  
\[
D:= \mathop{\rm diag}(d_1,\ldots,d_\m)
\]
with $d_j$ as in \eqref{defdk},
we can finally define
\[
h_j:=U D^{1/2} e_j,\quad j\in B.
\]
Since   $U:\C^\m\to{\mathcal W}^\perp$ is an isometry,
\begin{equation}
\label{hortog}
\langle h_j,h_k\rangle=
\langle D^{1/2}e_j,D^{1/2}e_k\rangle_{\C^\m}=
d_j \delta_{jk}\quad \forall j,k\in B,
\end{equation}
i.e. the vectors $h_j$ form an orthogonal system in
$\mathcal H$. Moreover  
$h_j\in {\mathcal W}^\perp$, that is
\begin{equation}
\label{hfperp}
\langle h_j,f_i\rangle=0\quad\forall j\in B,\quad
\forall i\in \{1,\ldots,\k\},
\end{equation}
hence by assumption (b) and
\eqref{defdk}
\begin{equation}
\label{fhorton}
\langle f_j+h_j,f_k+h_k\rangle=
\langle f_j,f_k\rangle +
\langle h_j,h_k\rangle
=\| f_j\|^2 \delta_{jk}+
d_j \delta_{jk}=\delta_{jk}\quad \forall j,k\in B,
\end{equation}
i.e. the vectors $f_j+h_j$ ($j\in B$) form an orthonormal
system as claimed in (i). Moreover, since
$U:\C^\m\to {\mathcal W}^\perp$ is an isometry,
for $j\in B$ we have
\begin{equation}
\label{estbathia}
\begin{aligned}
    & \| h_j-v_j\| =\left\| UD^{1/2}e_j-X e_j\right\|
    =\left\| UD^{1/2}e_j-UA^{1/2} e_j\right\|\\
    =& \left\| D^{1/2}e_j-A^{1/2} e_j\right\|_{\C^\m}
    \leq \| D^{1/2}-A^{1/2} \|_{\text{op}}
    \leq \| D-A \|_{\text{op}}^{1/2}
\end{aligned}
\end{equation}
where $\| \cdot \|_{\text{op}}$ denotes the operator norm
for $\m\times\m$ matrices (the last inequality,
valid for arbitrary positive semidefinite matrices $D$ and $A$, is classical,
see e.g. \cite[][Theorem X.1.1]{bhatia}). 

Now observe that each entry of the matrix $D-A$
is bounded by $C_2\eps$: for diagonal entries this
follows from \eqref{estdjvj}, while
for off-diagonal entries this follows from \eqref{estvjvk}
since $D$ is a diagonal matrix.
Therefore, the Hilbert--Schmidt norm
of $D-A$ is bounded  by $\m C_2\eps\leq \k C_2\eps$,
whence $\| D-A\|_{\text{op}}\leq \k C_2\eps$
which, combined with \eqref{estbathia},
gives
$\| h_j-v_j\|\leq (\k C_2\eps)^{1/2} $. As a consequence,
using \eqref{vcloseg} we obtain 
\begin{equation}
    \label{esthgB}
\| h_j-g_j\|\leq \|h_j-v_j\|+\|v_j-g_j\| \leq 
(\k C_2\eps)^{1/2}+
C_1^{1/2}\eps
\leq
C_3\eps^{1/2}
\quad\forall j\in B.
\end{equation}

Observe that, if $\m=\k$ (i.e. if $B=\{1,\ldots,\k\}$)
then the proof is complete with $C=C_3$,
due to \eqref{fhorton}, \eqref{hortog}, \eqref{hfperp}
and \eqref{esthgB}.

\medskip

\textbf{Step D.} Here we construct $h_j$ for $j\in S$,
under the assumption that $S\neq\emptyset$.

Consider the subspace ${\mathcal M}={\mathcal W}+
\mathop{\rm span}\{h_j\,:\, j\in B\}$. 
Since $\dim {\mathcal M}\leq \k+\m$
while $\mathop{\rm dim}{\mathcal H}\geq 2\k$, 
we have $\dim {\mathcal M}^\perp \geq \k-\m=|S|$,
hence
we may choose
an orthonormal system $(u_j)_{j\in S}\subset {\mathcal M}^\perp$
and define
\[
h_j:=\sqrt{d_j}\, u_j,\quad j\in S.
\]
These vectors are mutually orthogonal, 
orthogonal to $\mathcal W$, 
and 
orthogonal to the previously constructed vectors $h_i$,
$i\in B$. 
Since $B\cup S=\{1,\ldots,\k\}$ these properties,
combined with the corresponding properties \eqref{hortog},
\eqref{hfperp}, show that the vectors in
$(h_1,\ldots,h_\k)$ satisfy claim (ii) of the lemma.
Moreover, for $j\in S$, $\| h_j\|^2=d_j=1-\|f_j\|^2$, so that
the vectors $f_j+h_j$, $j\in S$, form an orthonormal system:
this, combined with \eqref{fhorton}, shows that 
claim (i) is also fulfilled (indeed, $\langle
f_i+h_i,f_j+h_j\rangle=0$ holds also 
whenever $i\in B$ and
$j\in S$, by the orthogonality properties established above).

In addition, 
using \eqref{estdjvj} and the definition of $S$ in
\eqref{defBS} we have
\[
\| h_j\|^2=d_j\leq  \left| d_j-\|v_j\|^2\right|+
\| v_j\|^2
< C_2\eps+\k C_2\eps=C_4\eps \quad\forall j\in S
\]
whence, using \eqref{vcloseg} and again the
definition of $S$ in \eqref{defBS},
\begin{equation*}
\|h_j-g_j\|
\leq \|v_j-g_j\| +\|v_j\|+\|h_j\|
\leq \sqrt{C_1}\eps+\sqrt{\k C_2\eps}+\sqrt{C_4\eps}
\leq C_5\sqrt{\eps}
\quad\forall j\in S.
\end{equation*}
Finally this estimate, combined with \eqref{esthgB}, proves
claim (iii), and the proof is complete.
\end{proof}

\section{Proof of the main theorem}\label{sec:proof of the main theorem}

In this section we prove Theorem \ref{th:main theorem for J_s}. We split the proof in several steps.

\textbf{Step I: Identification of the candidate.} The first step of the proof consists in identifying a candidate extremizer. To do so, consider a maximizing sequence $F^{(n)} = (f^{(n)}_1, \ldots, f_{\k}^{(n)}) \subset \Fock^{\k}$, where each $F^{(n)}$ is an orthonormal system. By the translation invariance \eqref{invarianceJ_s} and since Husimi functions vanish at infinity, we can assume that every Husimi function $u_{F^{(n)}}$ achieves its maximum at 0. Being a bounded sequence, we can extract a subsequence (still denoted with the same index) that is weakly convergent to a vector $F=(f_1, \ldots,f_{\k}) \in \Fock^{\k}$. Then, by weak lower semicontinuity of the norm we have $\|f_k\| \leq 1$ for every $k = 1,\ldots,\k$. Moreover, the vector $F$ is not zero. In fact, if we had $F^{(n)} \rightharpoonup 0$ then we would have
\begin{equation}\label{eq:convergence at 0}
    u_{F^{(n)}}(0) = \sum_{k=1}^{\k} |f_k^{(n)}(0)|^2 =\sum_{k=1}^{\k} |\langle f^{(n)}_k, k_0 \rangle|^2 \to 0
\end{equation}
but since every Husimi function achieves its maximum at $0$ this means $\|u_{F^{(n)}}\|_{L^{\infty}} \to 0$, but then
\begin{equation*}
    J_s(F^{(n)}) \leq s \|u_{F^{(n)}}\|_{L^{\infty}} \to 0,
\end{equation*}
hence $F^{(n)}$ would not be a maximizing sequence. Thanks to Lemma \ref{lem:change the limit to make it orthogonal}, we can apply a unitary matrix $U$ to the vector $F$ so that the vector $UF$ becomes an orthogonal (but not necessarily orthonormal) system without affecting either the corresponding Husimi function or the ones of the vectors $F^{(n)}$. Moreover, since $UF^{(n)} \rightharpoonup UF$ and since $UF^{(n)}$ is still an orthonormal system, by weak lower semicontinuity of the norm we have that every component of $UF$ has norm less or equal than 1. Therefore, we can directly assume that the weak limit $F$ forms an orthogonal system and that its components have norm less or equal than one. In addition, we can suppose that $F$ is ordered so that $\|f_1\| \geq \|f_2\| \geq \cdots \geq \|f_{\k}\|$. 

\textbf{Step II: Modify the tail.} The next step is to split the sequence $F^{(n)}$ into the surviving part ---i.e. with a nonzero limit--- and the tail ---i.e. with a zero limit--- and to modify the latter part in order to have better properties. To do so, let $ F^{(n)} - F \eqqcolon G^{(n)} = (g_1^{(n)}, \ldots, g_{\k}^{(n)}) \in \Fock^{\k}$, so that $G^{(n)} \rightharpoonup 0$. Then, we define
\begin{equation*}
    \varepsilon^{(n)} = \max_{j,k=1,\ldots,\k} | \langle f_j, g_k^{(n)} \rangle|.
\end{equation*}
Notice that since $G^{(n)} \rightharpoonup 0$ we have that $\varepsilon^{(n)} \to 0$ as $n \to \infty$. Since we have
\begin{itemize}
    \item $F+G^{(n)}$ is an orthonormal system;
    \item $F$ is an orthogonal system and $\|f_k\| \leq 1$ for every $k=1,\ldots,\k$,
\end{itemize}
we can apply Lemma \ref{lem:change the tail without changing the functional} with $\varepsilon=\varepsilon^{(n)}$ and replace $G^{(n)}$ with $H^{(n)} = (h^{(n)}_1, \ldots, h^{(n)}_{\k}) \in \Fock^{\k}$, where the elements of $H^{(n)}$ are pairwise orthogonal, orthogonal to all the elements of $F$ and such that $\|G^{(n)} - H^{(n)}\|_{\Fock^\k} \leq C \sqrt{\varepsilon^{(n)}}$ ---we highlight that the constant $C$ does not depend on $\varepsilon^{(n)}$ nor on the tuple $G^{(n)}$. Since $\varepsilon^{(n)} \to 0$ this last condition implies that $H^{(n)} \rightharpoonup 0$ and, thanks to Lemma \ref{lem:continuity of J_s}, that $F+H^{(n)}$ is still a maximizing sequence. 

\textbf{Step III: Split and rearrange the functional.} Now that we have identified the tail we can split the functional $J_s$. Fix $\varepsilon \in (0,1)$. Since $F+H^{(n)}$ is a maximizing sequence, for $n$ sufficiently large we have that
\begin{equation}\label{quasiopt}
    \S_\k(s) - \varepsilon \leq J_s(F+H^{(n)}) =  \int_{\Omega^{(n)}} u_{F+H^{(n)}}(z) \, dA(z),
\end{equation}
where we let $\Omega^{(n)} = \Omega_{F+H^{(n)}}(s) \subset \C$. Our goal is to uncouple the contributions coming from $F$ and $H^{(n)}$. We start noticing that it holds
\begin{equation*}
    u_{F+H^{(n)}}(z) \leq u_F(z) + u_{H^{(n)}}(z) + 2 \sqrt{u_F(z) u_{H^{(n)}}(z)} \quad z \in \C,
\end{equation*}
and therefore
\begin{equation}\label{eq:bound Husimi of the sum}
    J_s(F+H^{(n)}) \leq \int_{\Omega^{(n)}} u_{F}(z) + u_{H^{(n)}}(z) + 2\sqrt{u_F(z) u_{H^{(n)}}(z)} \, dA(z).
\end{equation}
Then, we observe that since $u_F$ is integrable there exists a compact $K \subset \C$ (independent of $n$) such that 
\begin{equation}\label{defK}
   \int_{\C\setminus K} u_F(z)\,dA(z)<\eps^2.
\end{equation}
In addition, since weak convergence $H^{(n)} \rightharpoonup 0$ 
implies uniform convergence of $u_{H^{(n)}}$ on compact sets, for $n$ large enough it holds
\begin{equation}
\label{estHonK}
   \int_{K} u_{H^{(n)}}(z)\,dA(z)<\eps^2.
\end{equation}
Combining \eqref{defK} and \eqref{estHonK} we obtain that
\[
\begin{aligned}
&\int_\C  \sqrt{u_F(z) u_{H^{(n)}}(z)}\,dA(z)= \int_K  \sqrt{u_F(z) u_{H^{(n)}}(z)}\,dA(z)+ \int_{\C\setminus K}  \sqrt{u_F(z) u_{H^{(n)}}(z)}\,dA(z)\\
&\leq \| F\| \left(\int_K u_{H^{(n)}}(z)\,dA(z)\right)^{1/2}+\| H^{(n)}\|\left(\int_{\C\setminus K}  u_F(z)\right)^{1/2} \leq 2\sqrt{\k} \eps,
\end{aligned}
\]
where we used the fact that $\| F\|\leq \sqrt\k$ and $\|H^{(n)}\|\leq\sqrt\k$ (recall \eqref{normaF}). Therefore, combining the previous estimate with \eqref{quasiopt} and \eqref{eq:bound Husimi of the sum} we obtain
\begin{equation}\label{splitJ}
\begin{aligned}
    J_s(F+H^{(n)}) &= \int_{\Omega^{(n)}} u_{F+H^{(n)}}(z) \, dA(z) \\
             &\leq  \int_{\Omega^{(n)}} u_{F}(z) + u_{H^{(n)}}(z) \, dA(z)
             +4\sqrt{\k} \varepsilon. 
\end{aligned}
\end{equation}
Finally, using Lemma \ref{lem:telescopic sum of the norms} we can rewrite the last integral so that we highlight the role of the mixed blocks, that is
\begin{equation}\label{eq:functional after block decomposition}
    \int_{\Omega^{(n)}} u_F(z)+u_{H^{(n)}}(z) \, dA(z) = \sum_{k=0}^{\k} a_k \int_{\Omega^{(n)}} u_{\widehat{M}_k^{(n)}}(z) \, dA(z)
\end{equation}
We point out that every mixed block $\widehat{M}_k^{(n)} \in \Fock^{\k}$ consists of at most $\k$ orthonormal functions.

\textbf{Step IV: Only $\widehat{F}$ survives.}
Our goal is to show that, as $n$ goes to infinity, the only block that gives a contribution in \eqref{eq:functional after block decomposition} is the one for $k=\k$, which is the block $\widehat{F}$. To do so, let $\j = \max\{k=1,\ldots,\k \colon f_k \neq 0\}$. Notice that since $F \neq 0$ then $\j \geq 1$. Combining the quasi optimality \eqref{quasiopt}, the splitting of the functional \eqref{splitJ}, the block decomposition \eqref{eq:functional after block decomposition} and the fact the blocks are composed by at most $\k$ orthonormal functions we have that
\begin{align*}
    \S_{\k}(s) - \varepsilon &\leq a_{\j} \int_{\Omega^{(n)}} u_{\widehat{M}_{\j}^{(n)}}(z) \, dA(z) + \sum_{0 \leq k \leq \k,\, k \neq \j} a_k \S_{\k}(s) + 4 \sqrt{\k} \varepsilon\\
                             &= a_{\j} \int_{\Omega^{(n)}} u_{\widehat{M}_{\j}^{(n)}}(z) \, dA(z) + (1-a_{\j}) \S_{\k}(s) + 4 \sqrt{\k} \varepsilon
\end{align*}
Since $a_{\j} \neq 0$ and since $\varepsilon$ is arbitrary, this implies that the mixed block $\widehat{M}_{\j}^{(n)}$ is asymptotically optimal, meaning that
\begin{equation}\label{eq:optimality block J}
    \limsup_{n \to \infty} \int_{\Omega^{(n)}} u_{\widehat{M}_{\j}^{(n)}}(z) \, dA(z) = \S_{\k}(s).
\end{equation}
We will prove by induction on $\k$ that $\j=\k$, thus proving that $\widehat{F}$ is optimal.

The base case $\k=1$ clearly holds because $\j$ is different from zero, therefore the supremum $\S_1(s)$ is attained for every $s>0$.

Then, we pass to the inductive step and we assume that the supremum $\S_{\m}(t)$ is attained for every $t \in [0,s]$ and for every $1 \leq \m < \k$. Assume also that $1 \leq \j < \k$, for otherwise there is nothing to prove. In this situation, the block that is asymptotically optimal is composed by two separate pieces
\begin{equation*}
    F_{\j} = (\widehat{f}_1, \ldots, \widehat{f}_{\j}) \in \Fock^{\j}, \quad H^{(n)}_{\k-\j} = (\widehat{h}^{(n)}_{k+1},\ldots,\widehat{h}^{(n)}_{\k}) \in \Fock^{\k-\j}.
\end{equation*}
We estimate the contribution of the two pieces separately. To estimate the contribution of $F_{\j}$ we notice that, recalling the definition of the compact set $K$ \eqref{defK}, it holds
\begin{equation*}
    \int_{\C \setminus K} u_{F_{\j}}(z) \, dA(z) \leq \int_{\C \setminus K} \dfrac{1}{\|f_{\j}\|^2} u_F(z) \, dA(z) \leq \dfrac{ \varepsilon^2}{\|f_{\j}\|^2},
\end{equation*}
and therefore
\begin{align}\label{eq:estimate on FJ}
    \int_{\Omega^{(n)}} u_{F_{\j}}(z) \, dA(z) &\leq \int_{\Omega^{(n)} \cap K} u_{F_{\j}}(z) \, dA(z) +  \dfrac{ \varepsilon^2}{\|f_{\j}\|^2} \nonumber \\
                                               &\leq \S_{\j}(|\Omega^{(n)} \cap K|) + \dfrac{ \varepsilon^2}{\|f_{\j}\|^2}.
\end{align}
To estimate the contribution of $H^{(n)}_{\k - \j}$ we notice that, by definition of $\j$, it holds $\|h^{(n)}_k\| = 1$ for every $k=\j+1, \ldots,\k$, and therefore from \eqref{estHonK} it follows
\begin{equation*}
    \int_{K} u_{H^{(n)}_{\k-\j}}(z) \, dA(z) \leq \int_K u_{H^{(n)}}(z) \, dA(z) \leq \varepsilon^2,
\end{equation*}
which in turn implies that
\begin{align}\label{eq:estimate on HK-J}
    \int_{\Omega^{(n)}} u_{H^{(n)}_{\k-\j}}(z) \, dA(z) &\leq \int_{\Omega^{(n)} \setminus K} u_{H^{(n)}_{\k-\j}}(z) \, dA(z) +  \varepsilon^2 \nonumber \\
                                               &\leq \S_{\k-\j}(|\Omega^{(n)} \setminus K|) + \varepsilon^2.
\end{align}
Combining the estimates \eqref{eq:estimate on FJ} and \eqref{eq:estimate on HK-J} we obtain
\begin{align*}
    \int_{\Omega^{(n)}} u_{\widehat{M}_{\j}^{(n)}}(z) \, dA(z) &=  \int_{\Omega^{(n)}} u_{F_{\j}}(z) \, dA(z) + \int_{\Omega^{(n)}} u_{H^{(n)}_{\k-\j}}(z) \, dA(z) \\
    &\leq \S_{\j}(|\Omega^{(n)} \cap K|) + \S_{\k-\j}(|\Omega^{(n)} \setminus K|)  + \left(1+\dfrac{1}{\|f_{\j}\|^2}\right) \varepsilon^2 \\
    &\leq \max_{0 \leq t \leq s} \left( \S_{\j}(t) + \S_{\k-\j}(s-t) \right) + \left(1+\dfrac{1}{\|f_{\j}\|^2}\right) \varepsilon^2
\end{align*}
Finally, combining the last estimate with the optimality of the block $\j$ \eqref{eq:optimality block J} and the fact that $\varepsilon$ is arbitrary we conclude
\begin{equation*}
    \S_{\k}(s) \leq \max_{0 \leq t \leq s} \left( \S_{\j}(t) + \S_{\k-\j}(s-t) \right).
\end{equation*}
However, from the inductive hypothesis we know that the suprema for indices less than $\k$ are always attained, and therefore from Lemma \ref{lem:convexity of the supremum} it follows that it must be $\j=\k$, for otherwise we would have a contradiction.

Having proved that $\j=\k$ we conclude that $\widehat{F}$ is optimal, thus concluding the proof.

\section{Extension to concentration operators for the wavelet transform}\label{sec:extension to the abstract setting}

Our main theorem along with its proof carries over to a more general setting, under a mild assumption (see \eqref{assumption:localization operators are positive} below) and with minor adjustments. In particular, we can extend our result to \emph{localization operators} associated with an abstract wavelet transform. While defining the setting, we will outline the parallels between the general setting and that in Fock space, whilst also highlighting any differences.

Let $\G$ be a locally compact Hausdorff group that is not compact. The group law will be denoted by the multiplicative notation. Let $\pi \colon \G \to \UU(\H)$ be an irreducible strongly continuous projective representation of $\G$ on an infinite-dimensional Hilbert space $\H$. We ask for $\G$ noncompact because otherwise the representation would be finite-dimensional (see, for example \cite[][Theorem 5.2]{follandbook}) and the problem would become trivial. We recall that strong continuity of the representation means that the map
    \begin{equation}\label{eq:strong continuity representation}
        x \in \G \mapsto \pi(x) \psi \in \H
    \end{equation}
    is continuous for every $\psi \in \H$. The representation $\pi$ is \emph{square-integrable} if there exists a vector $\psi \in \H \setminus \{0\}$ such that
    \begin{equation*}
        c_{\psi} = \int_{\G} |\langle \psi, \pi(x) \psi \rangle|^2 \, d\mu(x)
    \end{equation*}
    is finite, where $d\mu$ denotes the left Haar measure on $\G$. In this case, any normalized vector $\psi$ for which the constant $c_{\psi}$ is finite is called an \emph{admissible wavelet}. Upon renormalizing the Haar measure, we can suppose that $c_{\psi}=1$. Given an admissible wavelet, one can define the associated \emph{wavelet transform} of an element $f \in \H$ as
    \begin{equation}\label{eq:definition wavelet transform}
        W_{\psi}f(x) = \langle f, \pi(x) \psi \rangle, \quad x \in \G.
    \end{equation}
    With the chosen normalization, $W_{\psi}$ is an isometry from $\H$ to $L^2(\G)$. 
    \begin{remark}\label{remark:projective representation}
    The fact that $W_{\psi}$ is an isometry is a standard fact if $\pi$ is a ``true'' representation (see \cite[][Theorem 7.2]{wong}), but it is well-known also when $\pi$ is a genuine projective representation (see \cite[][Section 4]{christensen_projective}). Nevertheless, we want to give a brief and self-contained explanation of the reason why the theory works even in the projective case. In this situation, it holds
    \begin{equation*}
        \pi(xy) = \omega(x,y) \pi(x) \pi(y), \quad x,y \in \G,
    \end{equation*}
    where $\omega \colon \G \times \G \to U(1)$ is the continuous phase-factor associated with the projective representation. One can define a new group $\widetilde{\G}$, also called the Mackey obstruction group, whose underlying set is $\G \times U(1)$ and whose operation is defined as
    \begin{equation*}
        (x,\tau) \cdot (y,\lambda) = (xy, \overline{\omega(x,y)} \tau \lambda). 
    \end{equation*}
    Then, one can also define a \emph{true} representation $\widetilde{\pi}$ of $\widetilde{\G}$ on $\H$ as
    \begin{equation*}
        \widetilde{\pi}(x,\tau) = \tau \pi(x), \quad (x,\tau) \in \widetilde{\G},
    \end{equation*}
    and notice that also $\widetilde{\pi}$ is irreducible and strongly continuous. Next, we notice that if $\psi$ is an admissible wavelet for $\pi$, then for every $f \in \H$ and for every $(x,\tau) \in \widetilde{\G}$ it holds
    \begin{equation}\label{eq:relationship wavelet transform on G and its lift}
        \langle f, \widetilde{\pi}(x,\tau) \psi \rangle = \overline{\tau}W_{\psi}f(x).
    \end{equation}
    From this observation it follows that $\psi$ is admissible also for $\widetilde{\pi}$. Moreover, equation \eqref{eq:relationship wavelet transform on G and its lift} also implies that, up to choosing the correct normalization for the Haar measure of $U(1)$, it holds
    \begin{equation*}
        \|f\|^2 = \int_{\G \times U(1) } |\langle f, \widetilde{\pi}(x,\tau) \psi \rangle|^2 \, d\mu(x)d\tau = \int_{\G} |W_{\psi}f(x)|^2 \, d\mu(x),
    \end{equation*}
    hence $W_{\psi}$ is an isometry.
    \end{remark}

    Finally, given a measurable subset $\Omega \subset \G$, we can introduce the \emph{localization operator on $\Omega$}, denote by $T_{\Omega} \colon \H \to \H$ and defined weakly by
    \begin{equation}\label{eq:definition abstract localization operators}
        \langle T_{\Omega}f,g\rangle = \int_{\Omega} W_{\psi}f(x) \overline{W_{\psi}g(x)} \, d\mu(x), \quad f,g \in \H.
    \end{equation}
    We refer to \cite[][Chapter 12]{wong} for an introduction and for the main properties of localization operators. Once again, we point out the theory in \cite{wong} is developed for wavelet transforms associated with a ``true'' representation, but thanks to the observations in Remark \ref{remark:projective representation} (in particular thanks to \eqref{eq:relationship wavelet transform on G and its lift}) all the results in this setting transfer to the projective case. In particular, it holds that if $\Omega$ has finite measure, then $T_{\Omega}$ is nonnegative and compact (see \cite[][Proposition 13.1]{wong}), thus it admits a sequence of nonincreasing eigenvalues, that we shall denote by the same notation as in the Fock case. 
    \begin{remark}
        The Fock space fits into this framework by taking $\C$ with additive structure as underlying group $\G$, the Fock space $\Fock$ itself as Hilbert space $\H$ and the map $w \in \C \mapsto U_w \in \mathcal{U}(\Fock)$ as projective unitary representation, where $U_w$ are the unitary Weyl operators \eqref{eq:Weyl unitary operators}. Moreover, as suggested by the formula defining the Husimi function \eqref{eq:definition Husimi function}, choosing the reproducing kernel at 0  as admissible wavelet one obtains that the corresponding wavelet transform of $f \in \Fock$ is simply
        \begin{equation*}
            \langle f, k_z \rangle = f(z) e^{-\pi |z|^2/2}, \quad z \in \C.
        \end{equation*}
        Finally, it is clear that Toeplitz operators are the localization operators associated with the chosen wavelet transform, as for every $f,g \in \Fock$ it holds
        \begin{align*}
            \langle T_{\Omega}f,g \rangle = \int_{\Omega} f(z) \overline{g(z)} e^{-\pi |z|^2} \, dA(z) = \int_{\Omega} \left( f(z) e^{-\pi|z|^2/2}\right) \overline{\left(g(z) e^{-\pi|z|^2/2} \right)} \, dA(z).
        \end{align*}
    \end{remark}
    In this abstract setting we have the same mathematical objects as in the Fock case, namely localization operators and their eigenvalues. Therefore, also for localization operators associated with an abstract wavelet transform it makes sense to address the existence of a set maximizing the sum of the first $\k$ eigenvalues \eqref{eq:sum of eigenvalues} among all subsets $\Omega \subset \G$ such that $\mu(\Omega)=s$. In general, it may be possible that there is no set of prescribed measure $s$, since the measure $\mu$ could have atoms. However, we can safely suppose that the Haar measure $\mu$ is nonatomic, since this will be a consequence of Assumption \ref{assumption:localization operators are positive} (see Remark \ref{remark:assumption forces nonatomic} below).

    A careful analysis of the lemmas in Section \ref{sec:preliminary lemmas} and the proof in Section \ref{sec:proof of the main theorem} reveals that the arguments still work in the abstract setting with minor adjustments. In order to highlight the most critical points in the proof, we will go through its various steps, indicating those that, mutatis mutandis, remain unchanged and those where adjustments need to be made.

    The first subtle point is already in the reformulation of the problem in terms of the functional $J_s$. In fact, to define it we need to specify the expression of the set $\Omega_F(s)$. In the Fock case, $\Omega_F(s)$ is defined \eqref{defOmegaFs} as the unique superlevel set of the Husimi function of $F \in \Fock^{\k}$ with measure $s$. However, this definition does not necessarily apply in the general setting because it is not guaranteed that the Husimi function of a tuple $F = (f_1, \ldots, f_{\k}) \in \H^{\k}$, that is
    \begin{equation*}
        u_F = \sum_{k=1}^{\k} |W_{\psi}f_k|^2,
    \end{equation*}
    has an invertible distribution function. Nevertheless, we can still define $\Omega_F(s)$ as a set such that 
    \begin{equation}\label{eq:bathtub principle abstract setting}
        \int_{\Omega} u_F(x) \, d\mu(x) \leq \int_{\Omega_F(s)} u_F(x) \, d\mu(x)
    \end{equation}
    for every $\Omega \subset \G$ such that $\mu(\Omega) = s$. Such a set always exists by the bathtub principle \cite[][Theorem 1.14]{liebloss} but may not be unique.
    \begin{remark}
        Given $F \in \H^{\k} \setminus \{0\}$, let $\mu_F(t) = \mu(\{u_F>t\})$ ($t>0$) denote the distribution function of $u_F$. If $\mu_F$ takes the value $s$, then $\Omega_F(s)$ is unique and it is still the unique superlevel set of $u_F$ with measure $s$. On the other hand, if $\mu_F$ does not take the value $s$ then there exists $t^* \in (0,\infty)$ such that $\mu_F(t^*) < s$ and $\lim_{t \nearrow t^*} \mu_F(t) > s$ (notice that $\mu_F$ is right-continuous). This implies that $u_F$ does not have a superlevel set of measure $s$ and that the level set $\{u_F = t^*\}$ has positive measure. In this case, it holds that $\Omega_F(s) = \{u_F > t^*\} \cup \Omega$, where $\Omega$ is any subset of the level set $\{u_F=t^*\}$ such that $\mu(\Omega) = s - \mu_F(t^*)$ and for which \eqref{eq:bathtub principle abstract setting} holds. Since $\mu(\{u_F=t^*\}) > s-\mu_F(t^*)$, the choice of $\Omega$ is not unique.
    \end{remark}

    Next, up to adjusting the constants, one can see that all the lemmas proved in Section \ref{sec:preliminary results} still hold in the general setting, with the only exception of Lemma \ref{lem:convexity of the supremum}. In fact, in proving this lemma one uses the fact that Toeplitz operators associated with any set of positive measure are positive operators, while in the abstract setting localization operators are, in general, only nonnegative. Therefore, to adapt the proof to the general setting we have to make the following assumption:
    \begin{center}
        For every subset $\Omega \subset \G$ such that $\mu(\Omega)>0$ it holds
        \begin{equation}\label{assumption:localization operators are positive}
            \langle T_{\Omega} f, f \rangle = \int_{\Omega} |W_{\psi}f(x)|^2 \, d\mu(x) > 0, \quad \forall f \in \H \setminus \{0\}.
        \end{equation}
    \end{center}
    Under this assumption, Lemma \ref{lem:convexity of the supremum} holds also in the abstract setting. 
    \begin{remark}\label{remark:assumption forces nonatomic}
        If the Haar measure $\mu$ were atomic, then the group $\G$ would be discrete. Then, it is immediate to see that Assumption \ref{assumption:localization operators are positive} cannot hold. In fact, it suffices to consider any nonzero function $f \in \H$ that is orthogonal to $\pi(e)\psi$ (such a function exists because $\mathrm{dim}\H > 1$) and notice that
        \begin{equation*}
            \int_{\{e\}} |W_{\psi}f(x)|^2 \, d\mu(x) = 0
        \end{equation*}
        where $\mu(\{e\}) > 0$. 
    \end{remark}
    
    Turning to the proof in Section \ref{sec:proof of the main theorem}, one considers a maximizing sequence $F^{(n)} \in \H^{\k}$, where each $F^{(n)}$ is an orthonormal system. The first step of the proof remains unaltered. In particular, using the covariance property
    \begin{equation*}
        u_{\pi(y)}f(x) = u_f(y^{-1}x), \quad x,y \in \G
    \end{equation*}
    (the analogue of \eqref{eq:translation of the Husimi}) and the fact that Husimi functions vanish at infinity, one can translate each $F^{(n)}$ so that $u_{F^{(n)}}$ all achieve their maximum at the identity element of the group. This allows one to extract a weakly converging subsequence whose limit $F \in \H^{\k}$ is nonzero. Then, one can use Lemma \ref{lem:change the limit to make it orthogonal} to turn the tuple $F$ into an orthogonal system, without affecting nor its Husimi function nor the ones of $F^{(n)}$.

    The second step, where one carefully modifies the sequence $F^{(n)}$ so that its limit $F$ and the tail $H^{(n)}$ enjoy some orthogonality properties, still proceeds in the same way in the abstract setting.
    
    Similarly, the third step is just a matter of computation and carries over verbatim to the abstract setting. We just point out that also in the abstract setting weak convergence $H^{(n)} \rightharpoonup 0$ implies that $u_{H^{(n)}}$ converges uniformly to 0 on compact sets.

    Finally, the fourth step carries over without difficulties, as long as Lemma \ref{lem:convexity of the supremum} holds with strict inequality, which is the case under the assumption \eqref{assumption:localization operators are positive}. 
    
    Hence, in the end we conclude that the following holds.
    \begin{theorem}\label{th:main theorem abstract setting}
        Let $\k \in \N$ and $s>0$. Then, the supremum
        \begin{equation*}
            \sup \left\{ \sum_{k=1}^{\k} \lambda_k(\Omega) \colon \mu(\Omega) = s \right\}
        \end{equation*}
        is attained.
    \end{theorem}
    \subsection{Examples}\label{sec:examples}
    Theorem \ref{th:main theorem abstract setting} applies to several settings. In some of these, the existence result is new even for $\k=1$.
        \subsubsection{Multidimensional Fock space}\label{sec:multidimensional Fock space}
        The definition of the Fock space can be straightforwardly extended for entire functions on $\C^d$ for every $d \in \N$. We denote by $\Fock(\C^d)$ the Fock space for entire functions on $\C^d$. As in the case $d=1$, for $\k=1$ it was proved in \cite{nicolatilli_fk} that the first eigenvalue is maximized by balls in $\C^d$. However, differently from the case $d=1$, there is no analog of the result in \cite{nicola_riccardi_tilli_partial_sums} under some symmetry assumptions on the domains under consideration
        (in the multi-dimensional setting, the analog of a circularly symmetric domain would be a Reinhardt domain, see \cite{svela}).
        \subsubsection{Short-time Fourier transform}
        The Fock space can be seen as a particular example arising from the short-time Fourier transform. Given a normalized function $g \in L^2(\R^d) \setminus\{0\}$, the short-time Fourier transform of $f \in L^2(\R^d)$ with window $g$ is defined as
        \begin{equation}\label{eq:definition STFT}
            V_gf(x,\xi) = \int_{\R^d} f(t) \overline{g(t-x)} e^{-2\pi i \xi \cdot t} \, dt = \langle f, M_{\xi}T_xg \rangle_{L^2(\R^d)}\quad (x,\xi) \in \R^d \times \R^d,
        \end{equation}
        where $T_x$ denotes the translation operator by $x$ and $M_{\xi}$ the modulation operator by $\xi$ (see \cite{grochenig} for more details on the short-time Fourier transform). The short-time Fourier transform fits into the abstract framework taking $\G = \R^d \times \R^d$ with the additive structure, $\H = L^2(\R^d)$ and the projective representation given by time-frequency shifts:
        \begin{equation*}
            (x,\xi) \in \R^d \times \R^d \mapsto M_{\xi}T_x \in \mathcal{U}(L^2(\R^d)).
        \end{equation*}
        Then, if one chooses $g$ as admissible wavelet the corresponding wavelet transform is nothing but the short-time Fourier transform, as it is immediate to notice from the right-hand side of \eqref{eq:definition STFT}. Hence, Theorem \ref{th:main theorem abstract setting} applies to localization operators associated with the short-time Fourier transform, as long as for the window $g$ the assumption \eqref{assumption:localization operators are positive} is satisfied. For example, when $d=1$ if $g$ is a Hermite function (or a finite linear combination of Hermite functions) then the function $V_gf$ is real-analytic for every $f \in L^2(\R)$, and therefore assumption \eqref{assumption:localization operators are positive} is satisfied. Hence, our Theorem \ref{th:main theorem abstract setting} applies to the short-time Fourier transform with Hermite window. Even for $\k=1$, this result is new, with the only exception of the case where the window function is the normalized Gaussian (first Hermite function) $g(t) = 2^{d/4}e^{-\pi |t|^2}$. In fact, in this case one can prove that
        \begin{equation}\label{eq:relationship between STFT and Fock space}
            |V_{g}f(x,-\xi)|^2 = |\mathcal{B}f(z)|^2e^{-\pi |z|^2}, \quad z = x+i\xi \in \C^d
        \end{equation}
        where $\mathcal{B}f$ denotes the Bargmann transform of $f$, which is an entire function in the Fock space $\Fock(\C^d)$. The relationship \eqref{eq:relationship between STFT and Fock space}, along with the fact that the Bargmann transform is an isometry from $L^2(\R^d)$ into $\Fock(\C^d)$ (see \cite[][Proposition 3.4.1]{grochenig}), implies that one can equivalently consider localization operators for the short-time Fourier transform with Gaussian window and Toeplitz operators on the Fock space $\Fock(\C^d)$, which were already considered in the previous Section \ref{sec:multidimensional Fock space}. On the other hand, if one considers the case of higher order Hermite functions $g_n$ for $n \in \N$ (where $g_0$, that is the first Hermite function, is exactly the normalized Gaussian). In this case, a relationship analogous to \eqref{eq:relationship between STFT and Fock space} holds, where the Bargmann transform is replaced by the \emph{true polyanalytic Bargmann transform} of order $n$, which maps isometrically $L^2(\R)$ into the \emph{true polyanalytic Fock space} of order $n$ (see \cite{abreu_polyfock, abreu_feichtinger} for precise definitions of these objects). Whereas localization operators and concentration estimates have been studied in this particular setting \cite{abreu2021donoho}, we are not aware of any result concerning the optimality of sets in the spirit of our theorem.
        
        \subsubsection{Wavelet transform}
        Another class of interesting examples to which Theorem \ref{th:main theorem abstract setting} applies comes from ``usual'' wavelet transform.
        Consider a function $\psi \in L^2(\R) \setminus \{0\}$ satisfying the admissibility condition
        \begin{equation*}
            \int_{\R} |\widehat{\psi}(\xi)|^2 \, \dfrac{d\xi}{|\xi|} < \infty,
        \end{equation*}
        where the Fourier transform of $\psi$ is defined as
        \begin{equation*}
            \widehat{\psi}(\xi) = \dfrac{1}{\sqrt{2\pi}} \int_{\R} \psi(t) e^{-i \xi t} \, dt, \quad \xi \in \R.
        \end{equation*}
        For simplicity, we assume that the window $\psi$ belongs the the Hardy space
        \begin{equation*}
            H^2(\R) = \{f \in L^2(\R) \colon \widehat{f}(\xi) = 0 \text{ for a.e. } \xi < 0 \},
        \end{equation*}
        endowed with the norm inherited from $L^2(\R)$. Then, the wavelet transform of a function $f \in H^2(\R)$ with respect to the wavelet $\psi$ is defined as
        \begin{equation*}
            W_{\psi}f(b,a) = \int_{\R}f(t) \dfrac{1}{\sqrt{a}}\overline{\psi \left(\dfrac{t-b}{a} \right)} \, dt = \langle f, T_b D_a  \psi \rangle, \quad (b,a) \in \R \times \R_+.
        \end{equation*}
        where $T_b$ denotes the translation operator by $b$ and $D_a$ the dilation operator by $a$. This fits into the abstract framework taking the ``$ax+b$ group'' as group $\G$, $H^2(\R)$ as Hilbert space $\H$ and as unitary irreducible representation the one given by
        \begin{equation*}
            (b,a) \in \R \times \R_+ \mapsto T_b D_a  \in \mathcal{U}(H^2(\R)).
        \end{equation*}
        Notice that the restriction to the Hardy space is necessary, since this representation is not irreducible on $L^2(\R)$. 
        
        As for the short-time Fourier transform, assumption \eqref{assumption:localization operators are positive} is satisfied for some choices of the wavelet $\psi$. A particularly relevant choice for the wavelet is the so-called Cauchy wavelet $\psi_{\beta} \in H^2(\R)$ ($\beta>0$) in which case the wavelet transform of any $f \in H^2(\R)$ is analytic ---even more, it belongs to an appropriate weighted Bergman space. In this setting, it was proved \cite{ramos-tilli} that for $\k=1$ optimal sets are (hyperbolic) balls, while for $\k>1$ no result is known, even under symmetry assumptions on the sets taken into account. 
        
        \subsubsection{Paley--Wiener space and the Donoho--Stark conjecture} The last ---and arguably the most relevant--- example we give is the \emph{Paley--Wiener space}, which is defined by
        \begin{equation*}
            PW \coloneqq \{f \in L^2(\R) \colon \mathrm{supp}\, \widehat{f} \subseteq [-\pi,\pi] \}.
        \end{equation*}
        This is a well-known reproducing kernel Hilbert space, whose reproducing kernel at $x \in \R$ is
        \begin{equation*}
            \mathrm{sinc}(y-x) = \dfrac{\sin(\pi(y-x))}{\pi(y-x)} = T_x \mathrm{sinc}(y), \quad y \in \R.
        \end{equation*}
        In this setting, localization operators are also called \emph{time-frequency localization operators}. This class of operators has received a great deal of attention, starting from the works by Landau, Pollak and Slepian \cite{prolate_I, prolate_II, prolate_III, prolate_IV} right through to more recent results; a non-exhaustive list of examples is \cite{kulikov2024, kulikov2026sharp, kulikov_larsen2026sharp, marcerca_romero_speckbacher}. 
        
        The Paley--Wiener space and its time-frequency localization operators partially fit in the abstract framework taking $\G = \R$ with the additive structure, $\H=PW$ and the representation $x \in\R \mapsto T_x \in \mathcal{U}(PW)$, which is not irreducible (the subspace of functions with Fourier transform supported over the same subset of $[-\pi,\pi]$ is invariant). Nevertheless, the hypothesis of irreducibility is used to guarantee that the wavelet transform is an isometry, which is still the case in this setting if one takes the $\mathrm{sinc}$ function as wavelet. In fact, with this choice the wavelet transform is simply the identity:
        \begin{equation*}
            \langle f, T_x \mathrm{sinc} \rangle = f(x), \quad x \in \R.
        \end{equation*}
        Moreover, the assumption \eqref{assumption:localization operators are positive} is satisfied since every $f \in PW$ is analytic. Hence, Theorem \ref{th:main theorem abstract setting} in the Paley--Wiener setting reads as follows.
        \begin{theorem}\label{th:main theorem for PaleyWiener}
            Let $\k \in \N$ and $s>0$. Then, the supremum
            \begin{equation*}
                \sup \left\{ \sum_{k=1}^{\k} \int_{\Omega} |f_k(x)|^2 \, dx \colon |\Omega|=s,\ f_1, \ldots, f_{\k} \in PW \textrm{ orthonormal system}\right\}
            \end{equation*}
            is attained.
        \end{theorem}
        Of great importance is the case $\k=1$, which corresponds to the existence part of the statement of Theorem \ref{th:theorem PW introduction}. In fact, the first eigenvalue of the time-frequency localization over a set $\Omega$ is simply given by
        \begin{equation*}
            \lambda_1(\Omega) = \sup_{\|f\|_2=1} \int_{\Omega} |f(x)|^2 \, dx.
        \end{equation*}
        The maximization of $\lambda_1(\Omega)$ among all sets $\Omega\subset\R$ with the same measure is the topic of the well-known Donoho--Stark conjecture \cite{donoho_stark}, which states that intervals are the optimal sets. The conjecture was proved to be true first for $s<0.8/2\pi$ in \cite{donoho_stark_note_rearrangements}, later improved to $s<1/2\pi$ in \cite{baeza2023uncertainty}. Very recently, it was settled in full generality in \cite{abreu_speckbacher_donoho_stark}.
        
        Finally, the fact that the optimal set must be the union of finite intervals follows from Remark \ref{remark:optimal sets are bounded} and the fact that functions in Paley--Wiener are analytic. 
\section*{Acknowledgments}
F.N. is a fellow of the Accademia delle Scienze di Torino and a member of the Societ\`a Italiana di Scienze e Tecnologie Quantistiche (SISTEQ). 

\section*{Statements and Declarations}
\textbf{Competing interests.} The authors have no financial or non-financial interests to disclose.

\textbf{Data availability statement.} This manuscript has no associated data. 
	

\end{document}